\documentclass[mathpazo]{cicp}
\usepackage{bm}

\usepackage[latin1]{inputenc}
\usepackage{amsfonts,amsmath,amssymb,amsthm}
\usepackage{bbold}
\usepackage[american]{babel}
\usepackage[T1]{fontenc}
\usepackage{xspace}
\usepackage{xcolor}
\usepackage[all]{xy}
\usepackage[all]{nowidow}
\usepackage{paralist}
\usepackage{array,multirow}
\usepackage{tabularx}

\definecolor{darkgreen}{rgb}{0,0.6,0}
\definecolor{darkblue}{rgb}{0,0,0.6}

\providecommand{\begeq}[1]{\begin{align*}#1\end{align*}}
\newcommand{\secrevision}[1]{#1}
\newcommand{\AW}[1]{#1}
\newcommand{\revision}[1]{#1}
\newcommand{\norm}[1]{\left\lVert #1 \right\rVert}
\newcommand{\abs}[1]{\left\lvert #1 \right\rvert}
\newcommand{\N}{\mathbb{N}}

\DeclareMathOperator{\diag}{diag}

\newcommand\Tstrut{\rule{0pt}{2.6ex}}         
\newcommand\Bstrut{\rule[-0.9ex]{0pt}{0pt}}   
\newcolumntype{C}{>{\centering\arraybackslash}X}

\theoremstyle{plain}

\newtheorem{prop}{Proposition}[section]

\theoremstyle{definition}
\newtheorem{defi}{Definition}[section]

\theoremstyle{remark}
\newtheorem{rem}{Remark}[section]

\begin{document}

\title{A General Algorithm to Calculate the Inverse\\ Principal $p$-th Root of Symmetric Positive Definite Matrices}

\author[D.~Richters et al.]{Dorothee Richters\affil{1}, Michael Lass\affil{2,3}, Andrea Walther\affil{4}, Christian Plessl\affil{2,3}, and Thomas D.~K\"uhne\affil{3,5,6}\comma\corrauth}
\address{ \affilnum{1}\ Institute of Mathematics and Center for Computational Sciences,\\ Johannes Gutenberg University of Mainz,\\ Staudinger Weg 7, D-55128 Mainz, Germany\\
          \affilnum{2}\ Department of Computer Science, Paderborn University,\\
          \affilnum{3}\ Paderborn Center for Parallel Computing, Paderborn University,\\
          \affilnum{4}\ Institute of Mathematics, Paderborn University,\\
          \affilnum{5}\ Dynamics of Condensed Matter, Department of Chemistry, Paderborn University,\\
          \affilnum{6}\ Institute for Lightweight Design, Paderborn University,\\ Warburger Str.\ 100, D-33098 Paderborn, Germany}
\email{\tt tdkuehne@mail.upb.de}

\date{\today}

\begin{abstract}
We address the general mathematical problem of computing the inverse \mbox{$p$-th} root of a given matrix in an efficient way. A new method to construct iteration functions that allow calculating arbitrary $p$-th roots and their inverses of symmetric positive definite matrices is presented. We show that the order of convergence is at least quadratic and that adjusting a parameter $q$ leads to an even faster convergence. In this way, a better performance than with previously known iteration schemes is achieved. The efficiency of the iterative functions is demonstrated for various matrices with different densities, condition numbers and spectral radii.
\end{abstract}


\maketitle

\section{\label{sec:level1a}Introduction}
The first attempts to calculate the inverse of a matrix by the help of an iterative scheme were amongst others made by Schulz in the early thirties of the last century \cite{schu}. This resulted in the well-known Newton-Schulz iteration scheme that is widely used to approximate the inverse of a given matrix \cite{nr}. One of the advantages of this method is that for a particular start matrix as an initial guess, convergence is guaranteed \cite{pan}. The convergence is of order two, which is already quite satisfying, nevertheless there have been a lot of attempts to speed up this iteration scheme and to extend it to not only to calculate the inverse, but also the inverse $p$-th root of a given matrix \cite{high2, bini, iann1, iann2, smit}. This is an important task which has many applications in applied mathematics, computational physics and theoretical chemistry, where efficient methods to compute the inverse or inverse square root of a matrix are indispensable. An important example are linear-scaling 
electronic structure methods, which in many cases require to invert rather large sparse matrices in order to approximately solve the Schr\"odinger equation \cite{baroni1992, galli, PhysRevB.50.17611, ManolopoulosPalser, maslen1998, goed, PhysRevB.64.195110, kraj1, khaliullin2006, ceri1, haynes2008, ceri2, lin2009, bowler2012, kuehne, khaliullin2013, richMethod, mohr2014, PhysRevLett.112.046401, mohr2015, skylaris2016, mohr2017}. A related example is L\"owdin's method of symmetric orthogonalization \cite{loew, millam1997, niklasson2004, helgaker2007, kueh1, ochsenfeld2007, niklasson2008, we2009}, which transforms the generalized eigenvalue problem for overlapping orbitals into an equivalent problem with orthogonal orbitals, whereby the inverse square root of the overlap matrix has to be calculated. 

Common problems in the applications alluded to above, are the stability of the iteration function and its convergence. For $p \neq 1$, most of the iteration schemes have quadratic order of convergence, with for instance Halley's method being a rare exception \cite{hall, iann2, guo2}, whose convergence is of order three. Altman, however, generalized the Newton-Schulz iteration to an iterative method of inverting a linear bounded operator in a Hilbert space \cite{altm}. He constructed the so-called \emph{hyperpower method of any order of convergence} and proved that the method of degree three is the optimum one, as it gives the best accuracy for a given number of multiplications.

In this article, we describe a new general algorithm for the construction of iteration functions for the calculation of the inverse principal $p$-th root of a given matrix $A$. In this method, we have two variables, the natural number $p$ and another natural number $q\geq2$ that represents the \emph{order of expansion}. We show that two special cases of this algorithm are Newton's method for matrices \cite{iann1, iann2, guo2, smit, bini, hosk, guo1} and Altman's hyperpower method \cite{altm}.

The remainder of this article is organized as follows. In Section~\revision{\ref{sec:level2a}}, we give a short summary of the work presented in the above mentioned papers and show how we can combine this to a general expression in Section~\revision{\ref{sec:level3a}}. In Section~\revision{\ref{sec:numResults}}, we study the introduced iteration functions for both, the scalar case and for symmetric positive definite random matrices. We investigate the optimal order of expansion $q$ for matrices with different densities, condition numbers and spectral radii.

\section{\label{sec:level2a}Previous Work}
Calculating the inverse $p$-th root, where $p$ is a natural number, has been studied extensively in previous works. The characterization of the problem is quite simple. In general, for a given matrix $A\secrevision{\in\mathbb{C}^{n \times n}}$ one wants to find a matrix $B$ that fulfills $B^{-p} = A$. If $A$ is \revision{invertible}, one can always find such a matrix $B$, though $B$ \secrevision{may not be} unique. The problem of computing the $p$-th root of $A$ is strongly connected with the spectral analysis of $A$. If, for example, $A$ is a real or complex matrix of order $n$, with no eigenvalues on the closed negative real axis, $B$ \revision{always exists}~\cite{bini}. As we only deal with symmetric positive definite matrices, we can restrict ourselves to \revision{the} unique \revision{positive definite} solution, which is called the \emph{principal $p$-th root} and whose existence is guaranteed by the following theorem.

\begin{theorem}[Higham \cite{high1}, 2008]
Let $A \in \mathbb{C}^{n \times n}$ have no eigenvalues on $\mathbb{R}^-$. There is a unique $p$-th root $B$ of $A$ all of whose eigenvalues lie in the segment $\{z : -\pi/|p| < \arg(z) < \pi/|p| \}$, and it is a primary matrix function of $A$. We refer to $B$ as the principal $p$-th root of $A$ and write $B = A^{1/p}$. If $A$ is real then $A^{1/p}$ is real. \label{higham}
\end{theorem}

\begin{rem}
Here, $p < 0$ is also included, so that \textup{Theorem \ref{higham}} also holds for the calculation of inverse $p$-th roots.
\end{rem}

\noindent The calculation of such a root is usually done by the help of an iteration function, since the brute-force computation is computationally very demanding or even infeasible for large matrices. For sparse matrices, iteration functions can even reduce the computational scaling with respect to the size of the matrix because values in intermediately occurring matrices \revision{are often} truncated to retain sparsity.
However, one should always keep in mind that the inverse $p$-th roots of sparse matrices are in general not sparse anymore, but usually full matrices.

One of the most discussed iteration schemes for computing the $p$-th root of a matrix is based on Newton's method for finding roots of functions. It is possible to approximate a zero $\hat x$ of $f: \mathbb{R} \to \mathbb{R}$, meaning that we have $f(\hat x) = 0$, by the iteration
\begeq{x_{k+1} = x_k - \frac{f(x_k)}{f'(x_k)}.}
Here, $x_k \to \hat x$ for $k \to \infty$ if $x_0$ is an appropriate initial guess. If, for an arbitrary $a$, $f(x) = x^p-a$, then 
\begeq{x_{k+1} = \frac{1}{p} \left[ (p-1)x_k + ax_k^{1-p}  \right]}
and $x_k \to a^{1/p}$ for $k \to \infty$. One can also deal with matrices and study the resulting rational matrix function $F(X) = X^p-A$ \revision{\cite{hosk,smit}.}
This has been the subject of a variety of previous papers \cite{bini, pan, high2, iann1, iann2, smit, guo1, guo2, laki, psar, beni, petr, hosk, high1, hall}. It is clear that this iteration converges to the $p$-th root of the matrix $A$ if $X_0$ is chosen close enough to the true root.
Smith also showed that Newton's method for matrices has issues concerning numerical stability for ill-conditioned matrices~\cite{smit}, \revision{al}though this is not the topic of the present work.

Bini et al. \cite{bini} proved that the matrix iteration 
\begin{align}
B_{k+1} = \frac{1}{p} \left[(p+1) B_k - B_k^{p+1}A\right], \quad B_0 \in \secrevision{\mathbb{C}^{n \times n}} \label{eqn:q2p-1}
\end{align}
converges quadratically to the inverse $p$-th root $A^{-1/p}$ if \revision{$A$ is positive definite, $B_0A = AB_0$ and \mbox{$\norm{I-B_0^pA} < 1$}. Here, $\norm{\cdot}$ denotes some consistent matrix norm. Bini et al.\ further proved the following}
\revision{
\begin{prop}
\label{prop}
Suppose that all the eigenvalues of $A$ are real and positive. The iteration function (\ref{eqn:q2p-1}) with $B_0 = I$ converges to $A^{-1/p}$ if the spectral radius $\rho(A) < p+1$. If $\rho(A) = p+1$ the iteration does not converge to the inverse of any $p$-th root of $A$.
\end{prop}
}

In his work dated back to 1959, Altman described the hyperpower method \cite{altm}. Let $V$ be a Banach space \revision{with norm $\norm{\cdot}$}, $A: V \to V$ a linear, bounded and \revision{invertible} operator, while $B_0 \in V$ is an approximate reciprocal of $A$ satisfying $\norm{I - AB_0} < 1$. For the iteration 
\begin{align}
B_{k+1} = B_k(I + R_k + R_k^2 + \ldots + R_k^{q-1}), \label{eq:alt}
\end{align}
the sequence $(B_k)_{k \in \mathbb{N}_0}$ converges towards the inverse of $A$. Here, $R_k = I-B_k^pA$ is the $k$-th residual.
Altman proved that the natural number $q$ corresponds to the order of convergence of Eq.~\revision{(\ref{eq:alt})}, so that in principle a method of any order can be constructed. He defined the optimum method as the one that gives the best accuracy for a given number of multiplications and demonstrated that the optimum method is obtained for $q=3$.

To close this section, we recall some basic definitions, which are crucial for the next section.
In the following, the iteration function $\varphi: \mathbb{C}^{n\times n} \to \mathbb{C}^{n\times n}$ is assumed to be sufficiently often continuously differentiable.

\revision{
\begin{defi}
Let $A\in \mathbb{C}^{m\times n}$. The matrix norms $\norm{\cdot}_1$, $\norm{\cdot}_\infty$ and $\norm{\cdot}_2$ denote the norms induced by the corresponding vector norm, such that
\begin{align*}
\norm{A}_1 &= \max_{\norm{x}_1 = 1} \norm{Ax}_1 = \max_{1\leq j\leq n} \sum_{i=1}^m \abs{a_{ij}},\\
\norm{A}_\infty &= \max_{\norm{x}_\infty = 1} \norm{Ax}_\infty = \max_{1\leq i\leq m} \sum_{j=1}^n \abs{a_{ij}},\\
\norm{A}_2 &= \max_{\norm{x}_2 = 1} \norm{Ax}_2 = \sqrt{\lambda_\text{max}(A^* A)},
\end{align*}
where $\lambda_\text{max}$ denotes the largest eigenvalue and $A^*$ is the conjugate transpose of $A$.
\end{defi}
}

\begin{defi}
Let $\varphi: \mathbb{C}^{n\times n} \to \mathbb{C}^{n\times n} $ be an iteration function. The process
\begin{align}
\label{eq:phi} B_{k+1} = \varphi(B_k), \quad k = 0,1,\ldots
\end{align}
is called \emph{convergent to} $Z \in \mathbb{C}^{n\times n}$, if for all start matrices $B_0 \in \mathbb{C}^{n\times n}$ \revision{fullfilling suitable conditions},
we have $\norm{B_k-Z}_{\revision{2}} \to 0$ for $k \to \infty$. \label{defi1}
\end{defi}
\begin{defi}
A \emph{fixed point} $Z$ of the iteration function in
Eq.~\revision{(\ref{eq:phi})} is such that $\varphi(Z) = Z$ and is said
to be \emph{attractive} if $\norm{\varphi^{\prime}(Z)}_{\revision{2}} < 1$. \label{defi2}
\end{defi}

\begin{defi}
Let $\varphi: \mathbb{C}^{n\times n} \to \mathbb{C}^{n\times n}$ be an iteration function with fixed point $Z \in  \mathbb{C}^{n\times n}$.
\secrevision{Let $q \in \mathbb{N}$ be the largest number such that} for all start matrices $B_0 \in \mathbb{C}^{n\times n}$ 
\revision{fullfilling suitable conditions} we have 
\begin{align*} 
\norm{B_{k+1}-Z}_{\revision{2}} \leq c \norm{B_k-Z}_{\revision{2}}^q \textup{ for } k = 0,1,\ldots, 
\end{align*}
where $c >0$ is a constant with $c < 1$ if $q=1$. \secrevision{The iteration function in Eq.~\revision{(\ref{eq:phi})} is called \emph{convergent of order} $q$.} \label{defi3}
\end{defi}

\noindent We also want to remind a well-known theorem concerning the order of convergence.
\begin{theorem} Let \revision{$\varphi: \mathbb{C}^{n\times n} \to \mathbb{C}^{n\times n} $} be a function with fixed point \revision{$Z \in  \mathbb{C}^{n\times n}$}. If \revision{$\varphi$} is $q$-times continuously differentiable in a neighborhood of \revision{$Z$} with $q \geq 2$, then \revision{$\varphi$} has order of convergence $q$, if and only if
\revision{
\begin{align*}
\varphi^{\prime}(Z) = \ldots =  \varphi^{(q-1)}(Z) = 0 \textup{ and } \varphi^{(q)}(Z) \neq 0.
\end{align*}
}
\label{gautschi}
\end{theorem}
\begin{proof}The proof can be found in the book of Gautschi \cite[pp. 278--279]{gaut}. \end{proof}

\section{\label{sec:level3a}\secrevision{Generalization of the Problem}}
In this section, we present a new algorithm to construct iteration schemes for the computation of the inverse principal $p$-th root of a matrix, which contains Newton's method for matrices and the hyperpower method as special cases. Specifically, we propose a general expression to compute the inverse $p$-th root of a matrix that includes an additional variable $q$ as the order of expansion. In Altman's case, $q$ is the order of convergence, but as we will see, this does not hold generally for the iteration functions as constructed using our algorithm. Nevertheless, choosing $q$ larger than two often leads to an increase in performance, meaning that fewer iterations, matrix multiplications and therefore computational effort is required. In Section~\ref{sec:numResults} we discuss how $q$ can be chosen adaptively.

The central statement of this article is the following
\begin{theorem} \label{them}
Let $A \in \mathbb{C}^{n \times n}$ be an invertible matrix and 
$p, q \in \mathbb{N} \setminus \{0 \}$ with $q \ge 2$. 
Setting 
\begin{align}
  \varphi : \mathbb{C}^{n \times n} \rightarrow \mathbb{C}^{n \times n}, \quad X \mapsto \frac{1}{p} [ (p-1) X- ((I-X^p A)^q -I) X^{1-p} A^{-1}],
  \label{eq:phi1}
\end{align}
we define the iteration
\begin{equation}
  \label{eq:def_pro}
B_0 \in \mathbb{C}^{n \times n}, \qquad B_{k+1} = \varphi (B_k) = \frac{1}{p} \left[ (p-1) B_k- ((I-B^p_k A)^q -I) B^{1-p}_k A^{-1}\right]. 
\end{equation}
If $B_0 A=AB_0$ and for $R_0 = I-AB^p_0$ the inequality
\begin{equation}
\label{eq-1}
\left \|R_0^q-\frac{1}{p^p}\sum_{i=2}^p\binom{p}{i}p^{p-i}R_0^{i-1}\left(I-R_0\right)^{1-i}\left(I-R_0^q\right)^i\right\|_2=:c<1
\end{equation}
holds then one has
\[
B_k A=AB_k \quad\forall k\in\N\qquad \mbox{ and } \qquad \lim\limits_{k \rightarrow \infty} B_k = A^{-1/p}.
\]
If $p > 1$,  the order of convergence of the iteration defined by Eq.~\revision{\eqref{eq:def_pro}} is quadratic. If $p = 1$, the order of convergence of Eq.~\revision{\eqref{eq:def_pro}} is equal to $q$.
\end{theorem}
\begin{rem}
One can use the same formula to calculate the $p$-th root of $A$, where $p$ is negative or even choose $p \in \mathbb{Q}\setminus\{0\}$. However, this is not a competitive method because for negative $p$, one has to compute inverse matrices in every iteration step and for non-integers the binomial theorem and the calculation of powers of matrices becomes more demanding. \revision{In the following, we therefore always assume $p\in\mathbb{N}\setminus\{0\}$.}
\end{rem}
\begin{proof}
First, we show that $B_k A=AB_k$ holds for all $k \in \mathbb{N}$.
For this purpose, we rearrange
\begin{align*}
((I-X^pA)^q -I)&(X^pA)^{-1}  = \left(\sum^q_{i=0} \begin{pmatrix} q \\ i \end{pmatrix}
(-1)^i (X^pA)^i -I \right) (X^pA)^{-1}\\
& = \left( \sum^q_{i=1} \begin{pmatrix} q \\ i \end{pmatrix} (-1)^i (X^pA)^i \right) 
(X^pA)^{-1}
= \sum^q_{i=1} \begin{pmatrix} q \\ i \end{pmatrix} (-1)^i(X^pA)^{i-1}. 
\end{align*}
Now assume that $B_k A= A B_k$ holds for a fixed $k \in \mathbb{N}$. 
Then it follows for $B_{k+1}$ that 
\begin{align*}
B_{k+1} A &= \frac{1}{p}[(p-1)B_k-((I-B^p_k A)^q -I)B^{1-p}_k A^{-1}]A \\
&= \frac{1}{p} \left[(p-1)I- \sum^q_{i=1} \begin{pmatrix} q \\ i \end{pmatrix}(-1)^iA^{i-1}B^{p(i-1)}_k\right]B_k A \\
&= A \left( \frac{1}{p} \left[ (p-1)I- \sum^q_{i=1} \begin{pmatrix} q \\ i \end{pmatrix} (-1)^i A^{i-1} B^{p(i-1)}_k \right] B_k \right) = A B_{k+1}
\end{align*}
yielding the first part of the assertion. 
Next, we prove that with $R_k \equiv I- AB^p_k$ one has 
\begin{equation}
  \label{eq:Bk}
B_{k+1} = \frac{1}{p} B_k [pI+R_k+ \ldots + R^{q-1}_k].
\end{equation}
This will be shown by an induction over $q$. 
For $q=2$, one obtains
\begin{align*}
B_{k+1} &= \frac{1}{p} \left[(p-1)B_k-((I-B^p_kA)^2-I) B^{1-p}_k A^{-1}\right] \\
&= \frac{1}{p} \left[(p-1)B_k-(-2B^p_kA+B^{2p}_kA^2)B^{1-p}_kA^{-1}\right] 
= \frac{1}{p} B_k \left[(p-1)I+2I-B^p_kA \right] \\
&= \frac{1}{p} B_k [pI+R_k].
\end{align*}
Now assume that Eq.~\eqref{eq:Bk} holds for a fixed $q-1 \in \mathbb{N}$.
This yields for $q \in \mathbb{N}$ that 
\begin{align*}
B_{k+1} &= \frac{1}{p} [(p-1)B_k-((I-B^p_kA)^q-I)B^{1-p}_k A^{-1}] \\
&= \frac{1}{p} \left[(p-1)B_k-((I-B^p_kA)^{q-1}(I-B^p_kA)-I)B^{1-p}_k A^{-1}\right] \\
&= \frac{1}{p} \left[(p-1)B_k-((I-B^p_kA)^{q-1}-I)B^{1-p}_kA^{-1}  + (I-B^p_kA)^{q-1} B^p_kA B^{1-p}_k A^{-1} \right] \\
&= \frac{1}{p} B_k [pI+R_k + \ldots + R^{q-2}_k] + \frac{1}{p} (I-B^p_kA)^{q-1}B_k= \frac{1}{p} B_k [pI+R_k + \ldots + R^{q-2}_k] + \frac{1}{p} B_k R^{q-1}_k \\
&= \frac{1}{p} B_k [pI+R_k + \ldots + R^{q-1}_k]. 
\end{align*}
Now, everything is prepared to show the second assertion. 
Using the interchangeability of the matrices $B_0$ and $A$ as well as the binomial theorem, one has
\begin{align}
  R_{1}& = I - \frac{1}{p^p} (I-R_0)(pI+R_0 + \ldots + R^{q-1}_0)^p \label{eq:R11}\\
  & = I- \frac{1}{p^p} 
(I-R_0)\left(pI+ R_0(I-R_0)^{-1}(I-R_0^q)\right)^p\nonumber \\
& = I - \frac{1}{p^p} (I-R_0)\left(\sum_{i=0}^p\binom{p}{i}p^{p-i} R_0^i(I-R_0)^{-i}(I-R_0^q)^i\right)\nonumber\\
& = I \!-\! \frac{1}{p^p} (I-R_0)\left(\!p^pI+p\,p^{p-1}R_0(I-R_0)^{-1}(I-R_0^q)+\sum_{i=2}^p\binom{p}{i}p^{p-i} R_0^i(I-R_0)^{-i}(I-R_0^q)^i\!\right)\nonumber\\
& = R_0^{q+1}-\frac{1}{p^p}\sum_{i=2}^p\binom{p}{i}p^{p-i} R_0^i(I-R_0)^{1-i}(I-R_0^q)^i\nonumber\\
& = R_0\left(R_0^q-\frac{1}{p^p}\sum_{i=2}^p\binom{p}{i}p^{p-i} R_0^{i-1}(I-R_0)^{1-i}(I-R_0^q)^i\right) \label{eq:R1}
\end{align}
Taking norms on both sides, one obtains for $k=0$
\[
\| R_1 \|_2 \le  \| R_0 \|_2 \cdot\left\|R_0^q-\frac{1}{p^p}\sum_{i=2}^p\binom{p}{i}p^{p-i} R_0^{i-1}(I-R_0)^{1-i}(I-R_0^q)^i\right\|_2 = c \|R_0\|_2
\]
with $c$ as defined in Eq.~\eqref{eq-1}. 
Using once more an induction, one obtains
\[
\| R_{k+1} \|_2 < c \| R_k \|_2 
\]
yielding linear convergence of $R_k$ to the zero matrix.
Hence, $B_k$ must converge to $A^{-1/p}$.

To determine the convergence order, we calculate for the iteration given by Eq.~\eqref{eq:def_pro}
the first and second derivative, espectively, given by
\begin{align*}
\varphi^{\prime}(X) =~&  q (I-X^{p}A)^{q-1}   
+ \frac{p-1}{p} \left(\left((I-X^{p}A)^q-I \right) (X^pA)^{-1} + I \right)\\
\varphi^{(2)}(X) = &-pq(q-1)X^{p-1}A(I-X^{p}A )^{q-2} + q(1-p)X^{-1} (I-X^{p}A )^{q-1} \\
&+ (p-1)X^{-p-1}A^{-1}(I-X^{p}A )^q- (p-1)X^{-p-1}A^{-1}.
\end{align*}
This yields $\varphi^{\prime}(A^{-1/p}) = 0$ for every $p$. For the second derivative it holds that
\begeq{\varphi^{(2)}(A^{-1/p}) = 0 \quad \Leftrightarrow\quad p=1.}

\noindent It is also possible to show that $\varphi^{(j)}(A^{-1/p})=0$, if and only if $p=1$ for $j = 3, \ldots, q-1$, since 
\begeq{ \varphi^{(j)}(X) = \sum_{i=0}^{j} J_{i,j}(I-X^{p}A)^{q-i} + (p-1)J_{j} X^{1-p-j} A^{-1}, }
where $J_{i,j}$ and $J_j$ are both \revision{non-zero rational} numbers.
This implies that according to Theorem \ref{gautschi}, the convergence of iteration defined by Eq.~\revision{(\ref{eq:def_pro})} is exactly quadratic for any $q$ if $p \neq 1$. For $p=1$, we have shown that the order of convergence is identical with $q$.
\end{proof}

Comparing this result with previous conditions for convergence the required Eq.~\eqref{eq-1} looks
rather complicated.
To illustrate that this condition is really necessary,
Fig.~\ref{fig:illu_norm} shows for the scalar case the development of
the residual after one step, i.e., $r_1$, when the initial residual $r_0$ varies in the interval $[0,1]$.
\begin{figure}
  \begin{center}
    \includegraphics[width=4.5cm]{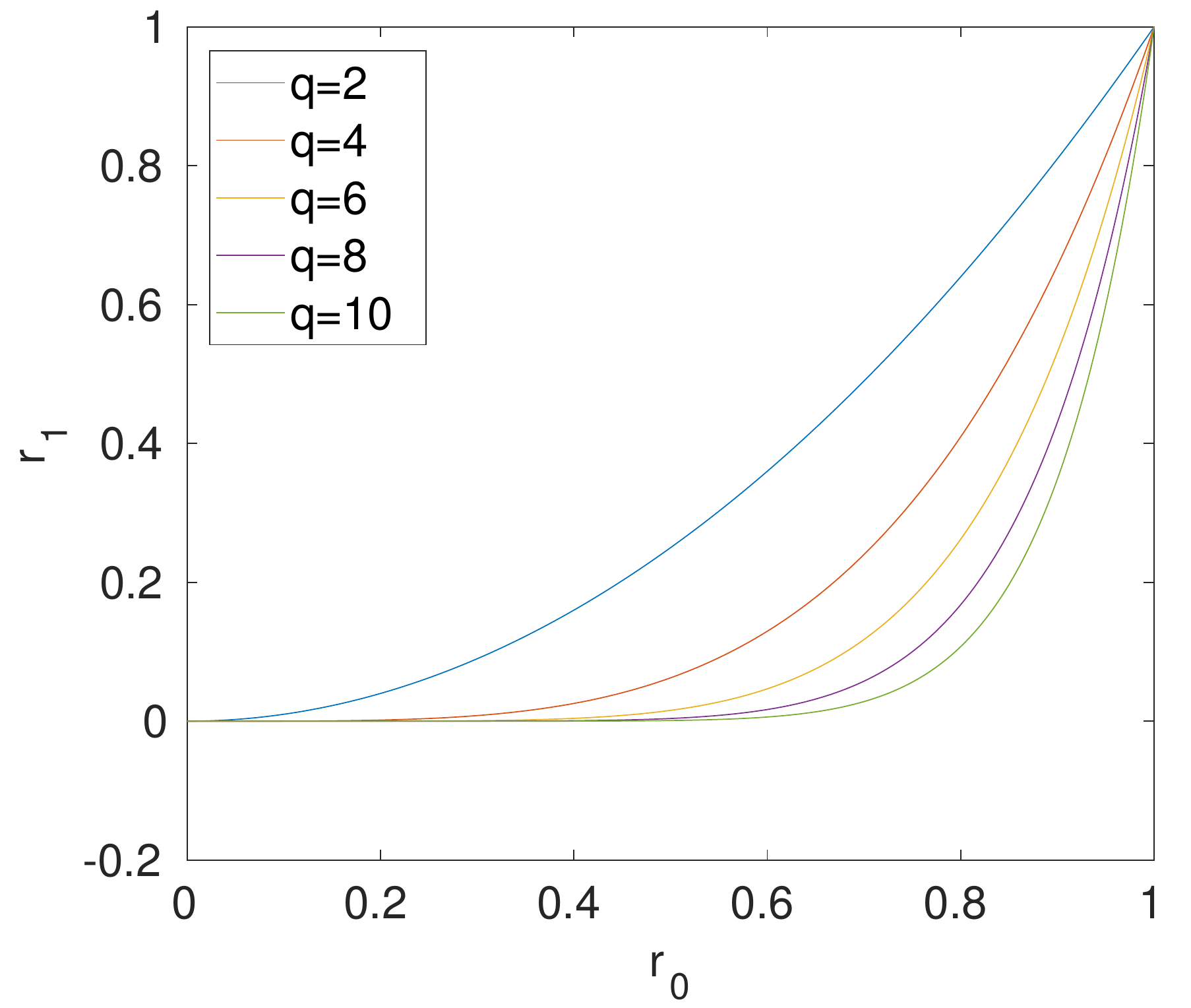} \quad
    \includegraphics[width=4.5cm]{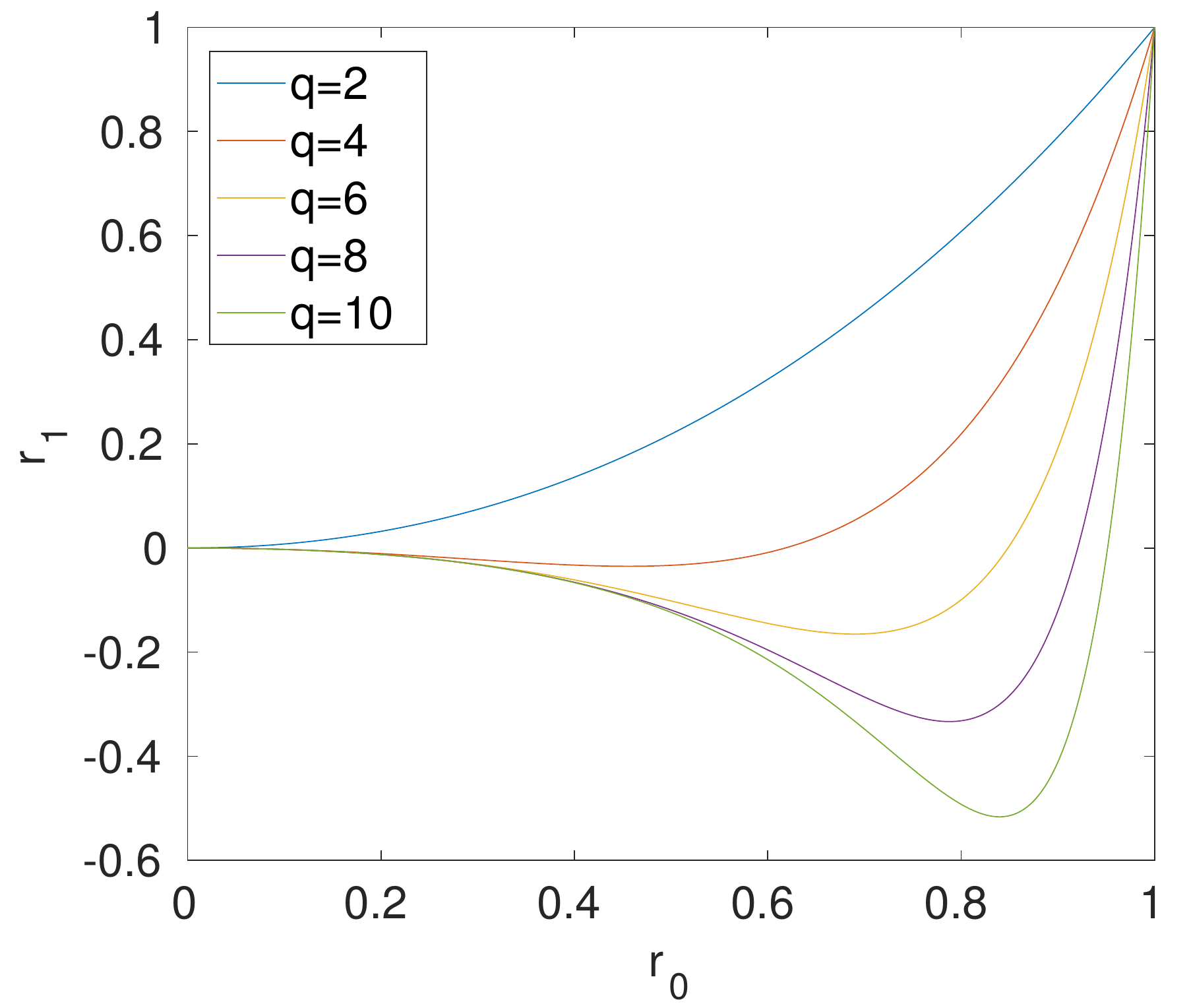} \quad
    \includegraphics[width=4.5cm]{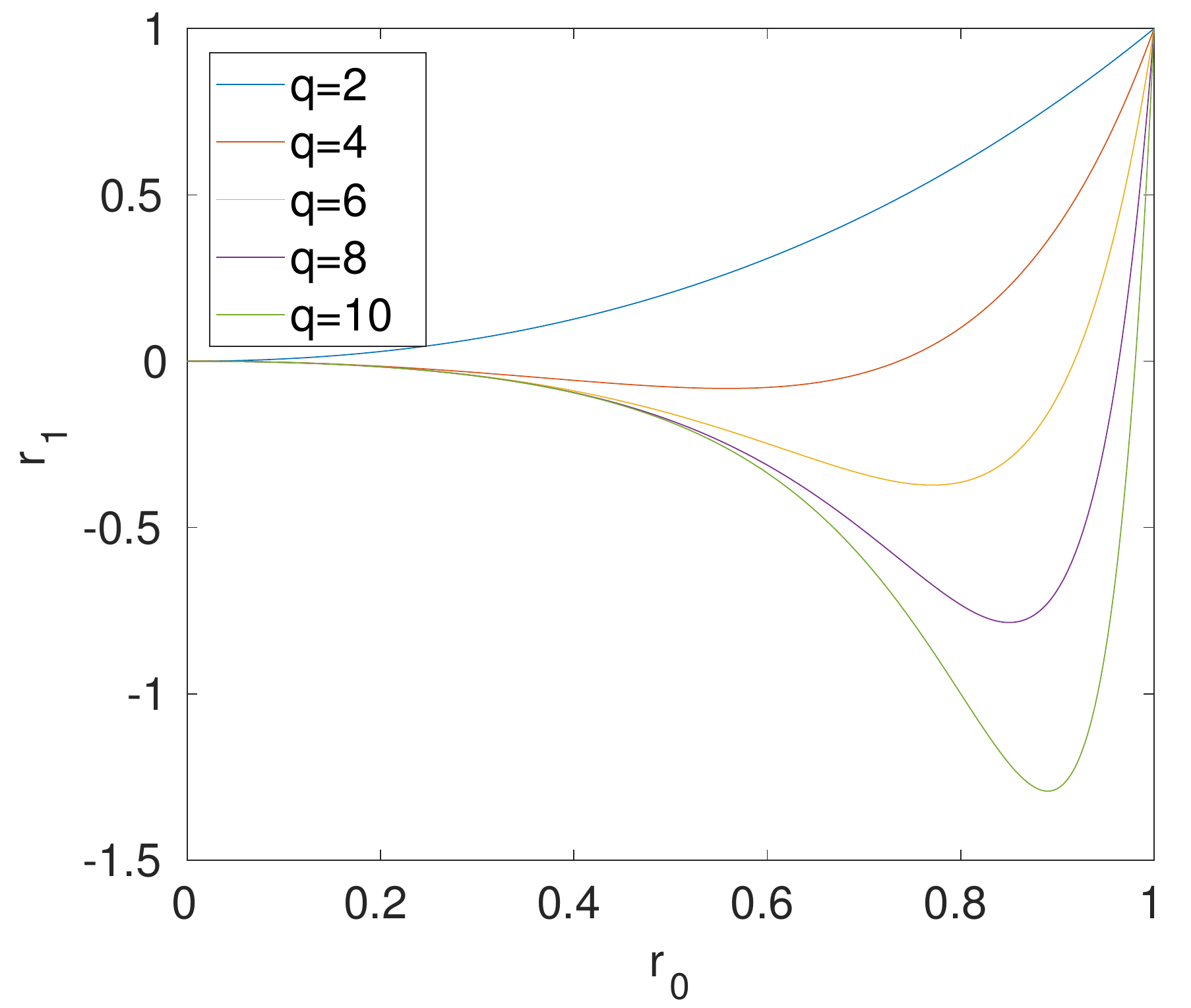}

   \quad $p=1$ \hspace*{4cm} $p=2$ \hspace*{4cm} $p=3$

    \includegraphics[width=4.5cm]{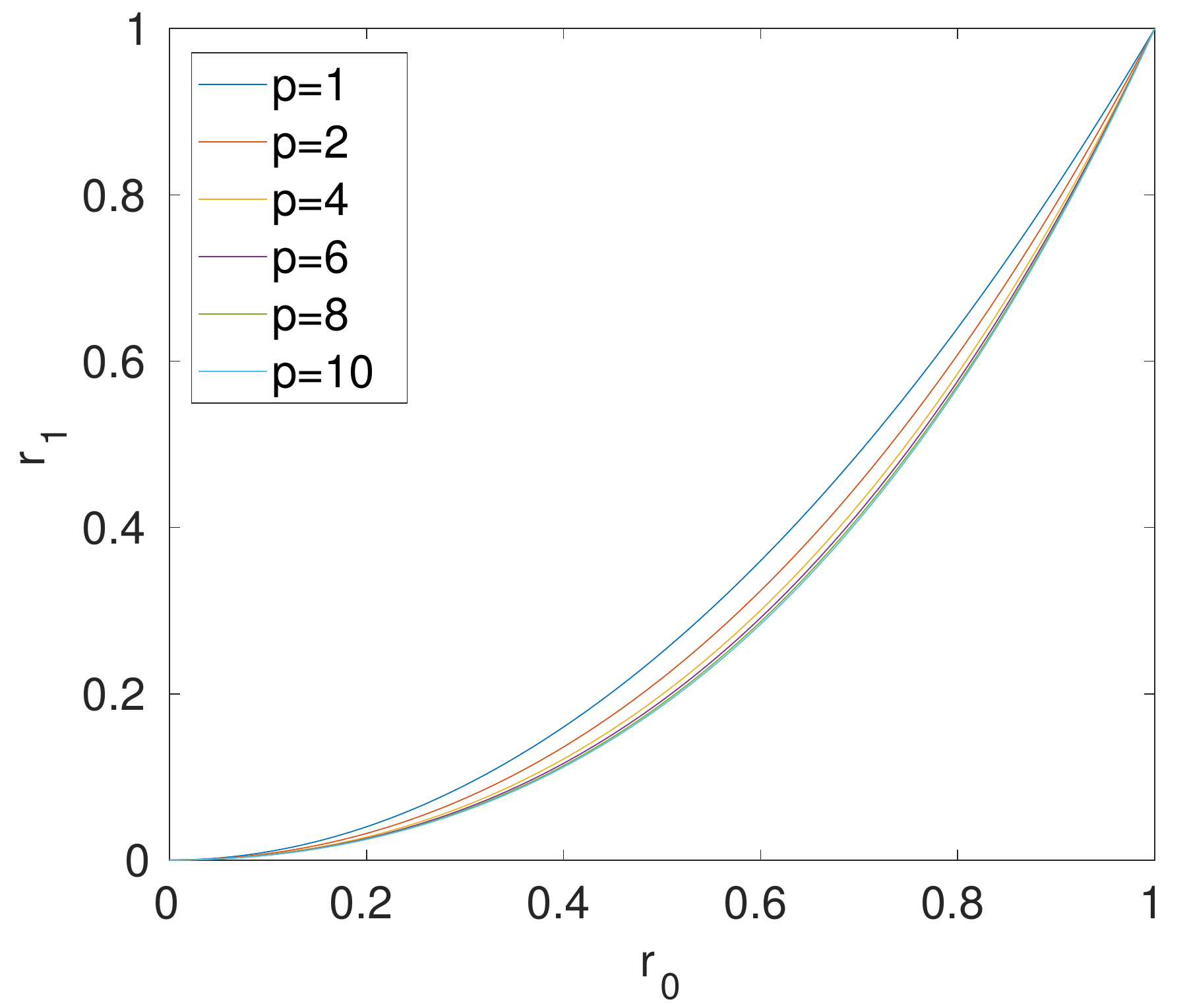} \quad
    \includegraphics[width=4.5cm]{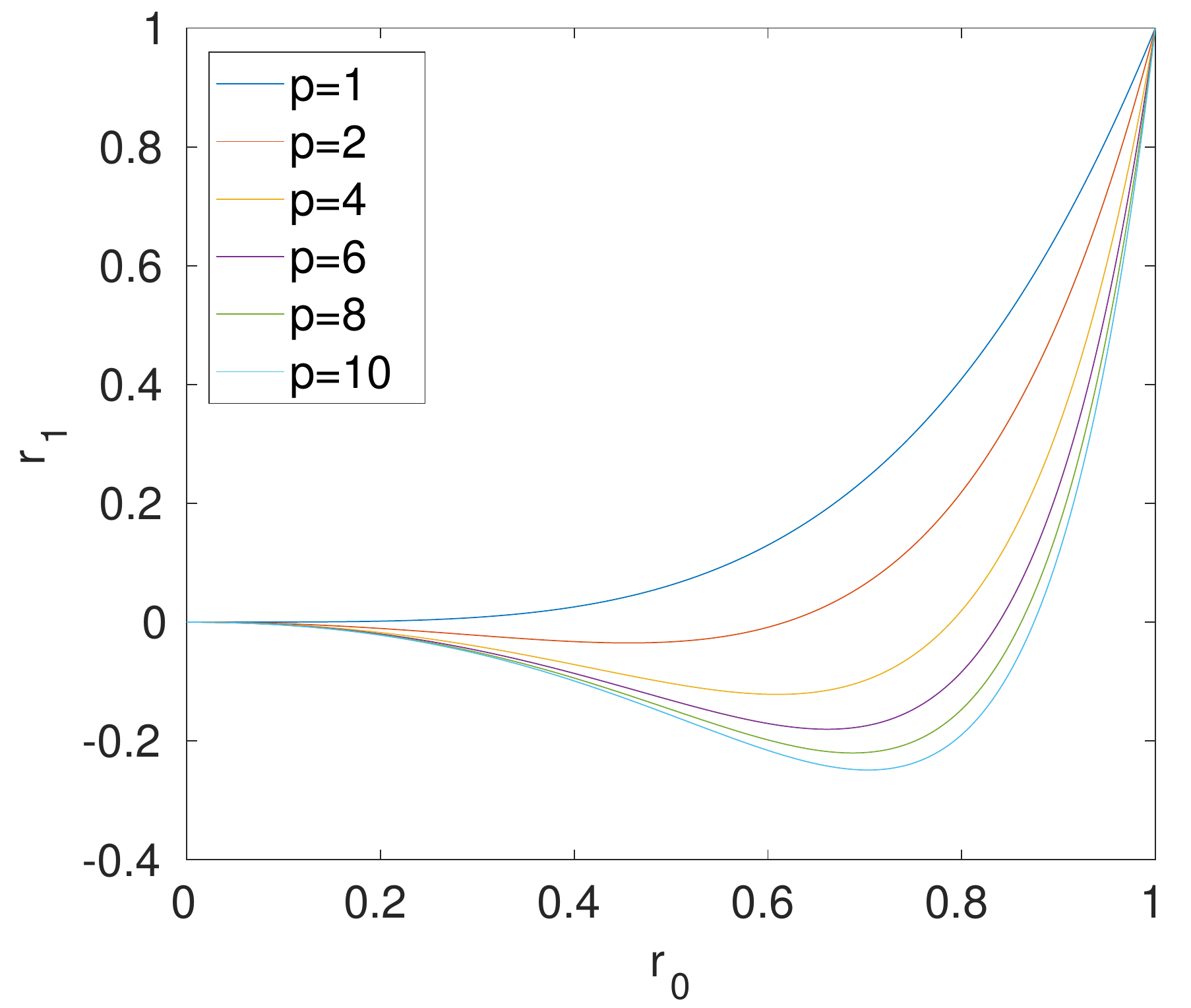} \quad
    \includegraphics[width=4.5cm]{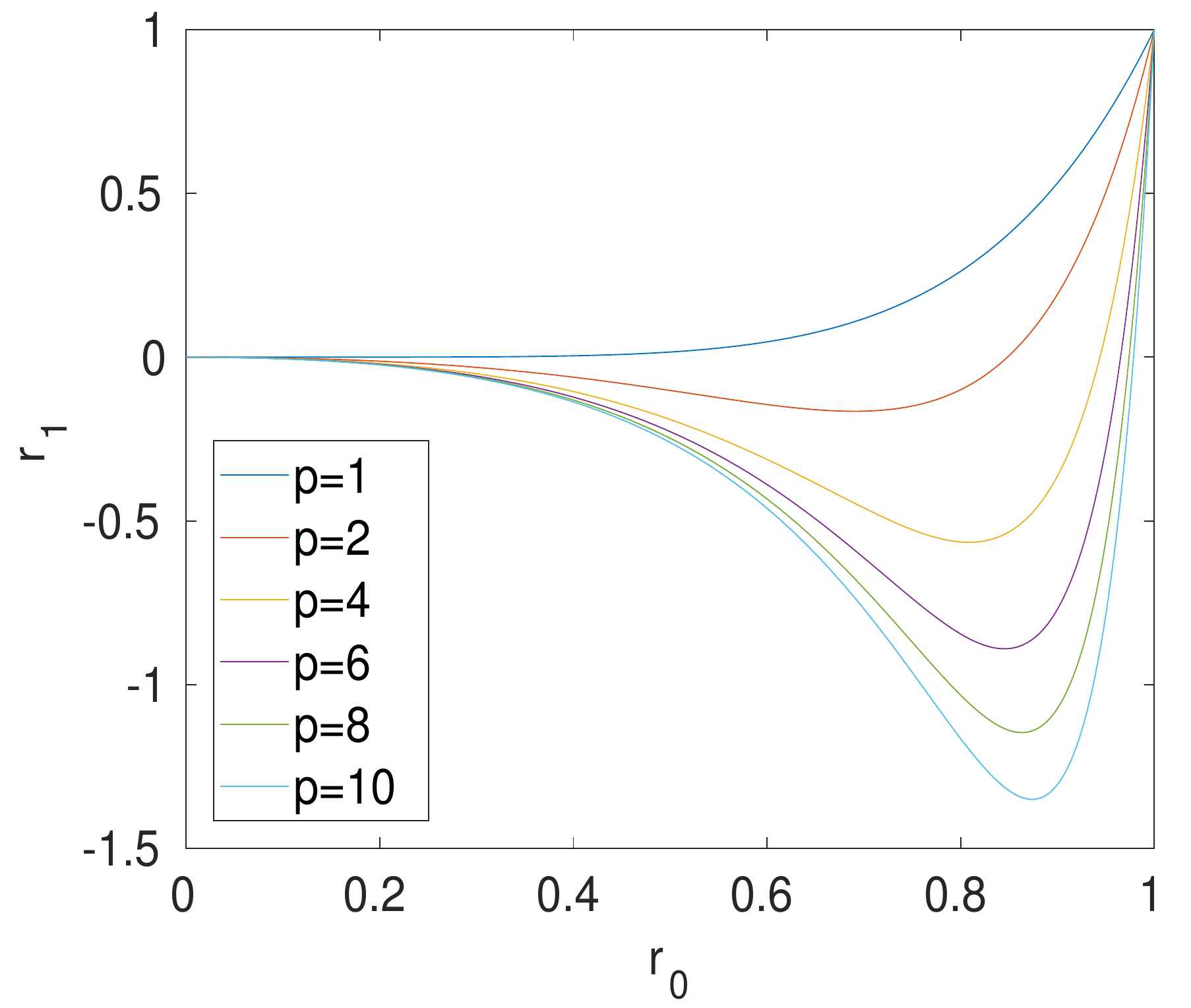}

   \quad $q=2$ \hspace*{4cm} $q=4$ \hspace*{4cm} $q=6$
  \end{center}
\caption{Norm of the second residual in dependence on the first residual}
\label{fig:illu_norm}
\end{figure}
It can be seen that the absolut value of the second residual can be larger than one even if the norm of the first residual is smaller than one if $p$ and $q$ are chosen in the wrong way. The rather complicated coniditon given in Eq.~\eqref{eq-1} ensures that the norm of the next residual decreases for all possible combinations of $p$ and $q\ge2$.  However, these illutrations suggest that for $p=1$ and arbitrary $q\ge2$ as well as for $q=2$ and arbitrary $p\ge1$ a simplification of the criterion given in Eq.~\eqref{eq-1} should be possible. For these cases one finds:
\begin{proposition}
\label{prop:simpler}
Let $A \in \mathbb{C}^{n \times n}$ be an invertible matrix.  Consider the iteration
defined by Eq.~\eqref{eq:def_pro} for $p, q \in \mathbb{N} \setminus \{0 \}$ and $q \ge 2$. 
If $p=1$ and $q\ge2$ or if $q=2$ and $p\ge1$, then the condition
\begin{equation*}
\left \|R_0^q-\frac{1}{p^p}\sum_{i=2}^p\binom{p}{i}p^{p-i}R_0^{i-1}\left(I-R_0\right)^{1-i}\left(I-R_0^q\right)^i\right\|_2=:c<1
\end{equation*}
for $R_0= I-AB^p_0$ can be simplified to
\begin{equation}
  \label{eq_eq-1a}
 \|R_0\|_2=: \tilde{c} <1
\end{equation}
\end{proposition}
\begin{proof}
For $p=1$ and $q\ge 2$, Eq.~\eqref{eq-1} reduces to
\begin{align*}
  \left\|R_0^q-\frac{1}{p^p}\sum_{i=2}^p\binom{p}{i}p^{p-i}R_0^{i-1}\left(I-R_0\right)^{1-i}\left(I-R_0^q\right)^i\right\|_2=\|R_0^q\|_2<1
  \quad\Leftrightarrow \quad\|R_0\|_2=: \tilde{c}<1.
\end{align*}
For $q=2$ and $p\ge1$,  one obtains from Eq.~\eqref{eq:R11} for $R_1$ the equation
\begin{align*}
  R_{1}& = I - \frac{1}{p^p} (I-R_0)(pI+R_0)^p,
\end{align*}
 when taking $q=2$ into account.
 Now, we will use another but equivalent reformulation for $R_1$ to show the assertion. Using the same reformulation as in \cite[Prop 6.1]{bini}, one has for $q=2$ and $p\ge1$ that
\begin{align*}
  R_{1}& = \sum_{i=2}^{p+1}a_iR_0^i\qquad\mbox{with}\qquad  a_i>0,\quad 2\le i\le p+1,\quad \sum_{i=2}^{p+1}a_i=1.
\end{align*}
Combining the first equality with the representation of $R_1$ given in Eq.~\eqref{eq:R1}, one obtains
\begin{align*}
  R_{1}& =R_0\left(R_0^q-\frac{1}{p^p}\sum_{i=2}^p\binom{p}{i}p^{p-i} R_0^{i-1}(I-R_0)^{1-i}(I-R_0^q)^i\right)  =
  R_0\left(\sum_{i=2}^{p+1}a_iR_0^{i-1}\right)
\end{align*}
yielding
\begin{align*}
  \|R_1\|_2 \le \|R_0\|_2\,\left\|\sum_{i=2}^{p+1}a_iR_0^{i-1}\right\|_2 \le \|R_0\|_2\left(\sum_{i=2}^{p+1}a_i\|R_0\|_2^{i-1}\right).
\end{align*}
If $\|R_0\|_2= \tilde{c} <1$ holds, it follows that
\begin{align*}
  \sum_{i=2}^{p+1}a_i\|R_0\|_2^{i-1} :=\hat{c} < \sum_{i=2}^{p+1}a_i = 1.
\end{align*}
Therefore, the inequality
\begin{align*}
   \|R_1\|_2 \le \hat{c} \|R_0\|_2,
\end{align*}
is valid, which yields convergence of $B_k$ as shown in the proof of Theo.~\ref{them}.
\end{proof}

\begin{table}
\begin{center}
  \begin{tabularx}{0.7\textwidth}{X|C|C|C|C|C|C|C|C|C}
    $p$ &  2 & 3 & 4 & 5 & 6 & 7 & 8 & 9 & 10\\ \hline
    $q$ & 15 & 8 & 7 & 6 & 6 & 5 & 5 & 5 & 5  
  \end{tabularx}
  \bigskip

  \noindent
  \begin{tabularx}{0.7\textwidth}{X|C|C|C|C|C|C|C|C}
    $q$ &  3 & 4 & 5 & 6 & 7 & 8 & 9 & 10\\ \hline
    $p$ &  $>20$ & $>20$ & $>20$ & 6 & 4 & 3 & 2 & 2   
  \end{tabularx}
\end{center}
  \caption{Highest value of second parameter (second line) for given first paramter (first line) for $n=1$.}
  \label{tab:overview}
\end{table}

To get an impression of the remaining cases, Tab.~\ref{tab:overview} lists the maximal possible values of the $q$ and $p$, respectively, when fixing $p$ and $q$, respectively, for the case $n=1$. These values were obtained by a simple numerical test varying $r_0$ in the whole interval $[0,1]$. For higher values of $n$, the numbers given in Tab.~\ref{tab:overview} represent an upper bound on the parameter values for which the simpler condition $\|R_0\|_2\le 1$ suffices.

With respect to the convergence order, as a trivial consequence we conclude that a larger value for $q$ does in general not lead to a higher order of convergence. However, in the following we perform numerical calculations using MATLAB \cite{matlab}, to determine the number of iterations, matrix-matrix multiplications, as well as the total computational time necessary to obtain the inverse $p$-th root of a given matrix as a function of $q$ within a predefined target accuracy. We find that a higher order of expansion than two within Eq.~\revision{(\ref{eq:phi1})} leads in almost all cases to faster convergence. Details are presented in Section~\ref{sec:numResults}.

We now show that our scheme includes as special cases the method of Bini et al.~\cite{bini} and the hyperpower method of Altman~\cite{altm}.
From now on, we deal with Eq.~\revision{(\ref{eq:def_pro})} for the definition of the matrices 
$B_k \in \mathbb{C}^{n \times n}$, i.e., define $B_{k+1} = \varphi(B_k)$. 
We assume that the start matrix $B_0$ satisfies 
$B_0A=AB_0$ and Eq.~(\ref{eq-1}). For $p=1$ and $q=2$, we get the already mentioned Newton-Schulz iteration that converges quadratically to the inverse of $A$
\begin{align*}
B_{k+1} &= - \left((I-B_kA)^2 -I \right) A^{-1} = (-(B_kA)^2 + 2B_kA)A^{-1}  \\
&= 2B_k - B_k^2A.
\end{align*}
For $q=2$ and any $p$, we get the iteration
\begin{align*}
B_{k+1} &= \frac{1}{p} \left[ (p-1) B_k - \left((I-B_k^{p}A)^2 -I\right) B_k^{1-p}A^{-1} \right] \nonumber \\
&= \frac{1}{p} \left[(p-1) B_k - \left( (B_k^{p}A)^2 - 2B_k^{p}A\right) B_k^{1-p}A^{-1} \right] \nonumber \\
&= \frac{1}{p} \left[(p-1) B_k - \left( B_k^{p+1}A - 2B_k\right) \right] \nonumber \\
& = \frac{1}{p} \left[(p+1) B_k - B_k^{p+1}A\right].
\end{align*}
This is exactly the matrix iteration in Eq.~\revision{(\ref{eqn:q2p-1})} that has been discussed in the work of Bini et al.~\cite{bini}. In both cases, Prop.~\ref{prop:simpler} shows that our condition Eq.~\eqref{eq-1} on $R_0$ is equivalent to the condition $\|R_0\|<1$ used for the convergence analysis of these methods in the corresponding original papers \cite{nr,bini}.

We now proceed by dealing with the iteration formula in the case $p=1$ and show that it converges faster for higher orders $(q>2)$. For that purpose, we take Eq.~\revision{(\ref{eq:def_pro})} and calculate for $p=1$
\begin{align*} 
B_{k+1} &= \frac{1}{p} \left[ (p-1) B_k - \left((I-B_k^{p}A)^q - I\right) B_k^{1-p}A^{-1} \right] \nonumber \\
&\overset{p=1}{=} \left[I - (I-B_kA)^q\right]A^{-1}.
\end{align*}
We now prove that this is convergent of at least order $q$ in the sense of Definition \ref{defi3}.
\begin{align*}
\norm{B_{k+1} -A^{-1}}_{\revision{2}} &= \norm{(I-(I-B_kA)^q)A^{-1} - A^{-1} }_{\revision{2}} \nonumber \\
&= \norm{(I-B_kA)^qA^{-1}}_{\revision{2}} \nonumber \\
&= \norm{(I-B_kA)^qA^{-1}A^{-q}A^q}_{\revision{2}} \nonumber \\
&\leq \norm{A}_{\revision{2}}^{q-1}\cdot\norm{(I-B_kA)^q(A^{-1})^q}_{\revision{2}} \nonumber \\
&\leq \norm{A}_{\revision{2}}^{q-1}\cdot\norm{A^{-1}-B_k}_{\revision{2}}^q.
\end{align*}
Next, we show why the iteration function in Eq.~\revision{(\ref{eq:def_pro})} coincides for $p=1$ with Altman's work on the hyperpower method \cite{altm}. We have already shown oin the proof of Theo.~\ref{them} that
\begin{align}
B_{k+1} &= \frac{1}{p}B_k\left[(p-1) + (R_k^{q-1}+R_k^{q-2}+\ldots+R_k+I)\right]=\frac{1}{p}B_k\left[pI + \left(\sum_{j=1}^{q-1}R_k^j\right) \right]. \label{eqn:Xk1}
\end{align}
Altman, however, proved convergence of any order for the iteration scheme in Eq.~\revision{(\ref{eq:alt})}, i.e.\revision{,}
\[ B_{k+1} = B_k(I + R_k + R_k^2 + \ldots + R_k^{q-1}), \quad B_0 \in V \]
when calculating the inverse of a given linear, bounded and \revision{invertible} operator $A \in V$.
If we take $\mathbb{R}^{n \times n}$ for the Banach space $V$, this is identical to Eq.~\revision{(\ref{eqn:Xk1})} with $p=1$. Once more,  Prop.~\ref{prop:simpler} shows that our condition Eq.~\eqref{eq-1} is equivalent to the one used by Altman in \cite{altm}.

\section{\label{sec:numResults}Numerical Results}
Even if the mathematical analysis of our iteration function results in the awareness that, except for $p=1$, larger $q$ \revision{do} not lead to a higher order of convergence, we conduct numerical tests by varying $p$ and $q$. Concerning the matrix $A$, whose inverse $p$-th root should be determined, we take real symmetric positive definite random matrices with different densities and condition numbers. We do this using MATLAB and quantify the number of iterations (\#it) \revision{and} matrix-matrix multiplications (\#mult) until the calculated inverse $p$-th root is close enough to the true inverse $p$-th root.

\subsection{The Scalar Case}
First, we assess our program for the scalar case. As the computation of roots of matrices is strongly connected with their eigenvalues, it is logical to study the formula for scalar quantities $\lambda$ first, which corresponds to $n=1$ in Eq.~\revision{(\ref{eq:phi1})}. We take values $\lambda$ varying from $10^{-9}$ to $1.9$ and choose $b_0 = 1$ as the start value. For that choice of $b_0$, we have guaranteed convergence for $\lambda \in (0, 2)$ \AW{for $q\le 10$ as illustrated in Fig.~\ref{fig:illu_norm}}. We calculate the inverse $p$-th root of $\lambda$, i.e.\revision{,} $\lambda^{-1/p}$, where the convergence threshold $\varepsilon$ is set to $10^{-8}$. In most cases, already a larger value for $\varepsilon$ yields sufficiently accurate results, but to see differences in the computational time, we choose an artificial small threshold.
\revision{While running the iteration scheme, we count the number of iterations and the total number of multiplications until convergence.}
\revision{In contrast to the matrix case, we do not use the norm of the residual $r_k = 1 - b_k^p \cdot \lambda$ as a criterion for the convergence threshold, but the \revision{error} between the $k$-th iterate and the correct inverse $p$-th root, thus $\tilde r_k = b_k - \lambda^{-1/p}$.}
This is due to the fact that in the scalar case, one can easily get $\lambda^{-1/p}$ by a straightforward calculation. Note that we have not necessarily $|\tilde r_k| < 1$, but only $|r_k| < 1$. \revision{In order} to better distinguish the scalar from the matrix case, we write in the scalar case $b_k, \lambda$, and $r_k$ instead of $B_k, A$, and $R_k$.

We choose a rather wide range for $q$ to see the influence of this value on the number of iterations and \revision{multiplications}. In agreement with our observations in the matrix case, which we describe in the next subsection, there is a correlation between the choice of $q$, the number of multiplications \revision{and the number of iterations}. In the following, we elaborate \revision{on the optimal choice of $q$}.

\begin{table}[bt]
\centering
 \begin{minipage}[t]{0.475\textwidth}
  \centering
   \caption{Numerical results for $p=2$, $\lambda = 1.5$. Optimal values are shown in bold.\vspace{5mm}}
   \label{lamda15}
 \begin{tabular}{rrr}
\hline\hline
 $q$ & \#it & \#mult\Tstrut\Bstrut\\
 \hline
2 & 5 & 37\Tstrut\\
3 & 4 & 34 \\
\textbf{4} & \textbf{3} & \textbf{29} \\
5 & 4 & 42 \\
6 & \textbf{3} & 35 \\
7 & 4 & 50 \\
8 & 4 & 54\Bstrut\\
\hline\hline
 \end{tabular}
\end{minipage}\begin{minipage}[t]{0.05\textwidth}\hfill\end{minipage}\begin{minipage}[t]{0.475\textwidth}
   \centering
   \caption{Numerical results for $p=2$, $\lambda = 10^{-9}$. Optimal values are shown in bold.\vspace{5mm}}
   \label{lowev}
 \begin{tabular}{rrr}
\hline\hline
 $q$ & \#it & \#mult\Tstrut\Bstrut\\
 \hline
2 & 27 & 191\Tstrut\\
3 & 17 & 138 \\
4 & 14 & 128 \\
5 & 12 & \textbf{122} \\
6 & 11 & 123 \\
\textbf{7} & \textbf{10} & \textbf{122} \\
8 & \textbf{10} & 132\Bstrut\\
\hline\hline
 \end{tabular}
\end{minipage}
\end{table}
Typically, the number \revision{of necessary multiplications} is minimal for the smallest $q$, which entails the lowest number of required iterations.
However, there are cases where the number of iterations is not steadily decreasing for higher $q$, as for example when computing the inverse square root of $\lambda = 1.5$ (Table \ref{lamda15}). Nevertheless, we conclude that the best choice is $q=4$, as we have the lowest number of \revision{both} iterations and multiplications.
In other cases, however, the trend is much simpler. For example, when computing the inverse square root of $\lambda = 10^{-9}$, while varying $q$ from 2 to 8, we clearly obtain the most efficient result for $q = 7$, where the number of iterations \revision{and the number of multiplications are} minimal (Table \ref{lowev}).

In most of the cases, it is \textit{a priori} not apparent which value for $q$ is the best for a certain tuple $(p, \lambda)$. It is possible that the lowest number of iterations is attained for really large $q$, i.e.\revision{,} $q > 20$. To present a general rule of thumb, we pick $q$ from $3$ to $8$ as this usually gives a good and fast approximation of the inverse $p$-th root of a given $\lambda \in (0,2)$. Hereby, we observe that for values close to $1$, the optimal choice is in most, but not all cases, $q=3$ and for values close to the borders of the interval, mostly $q=6$ is the best choice. This shows that the further the value of $\lambda$ is away from our start value $b_0 = 1$, the more important it is to choose a larger $q$.

\begin{figure}[bt]
\centering
\includegraphics[width=.7\textwidth]{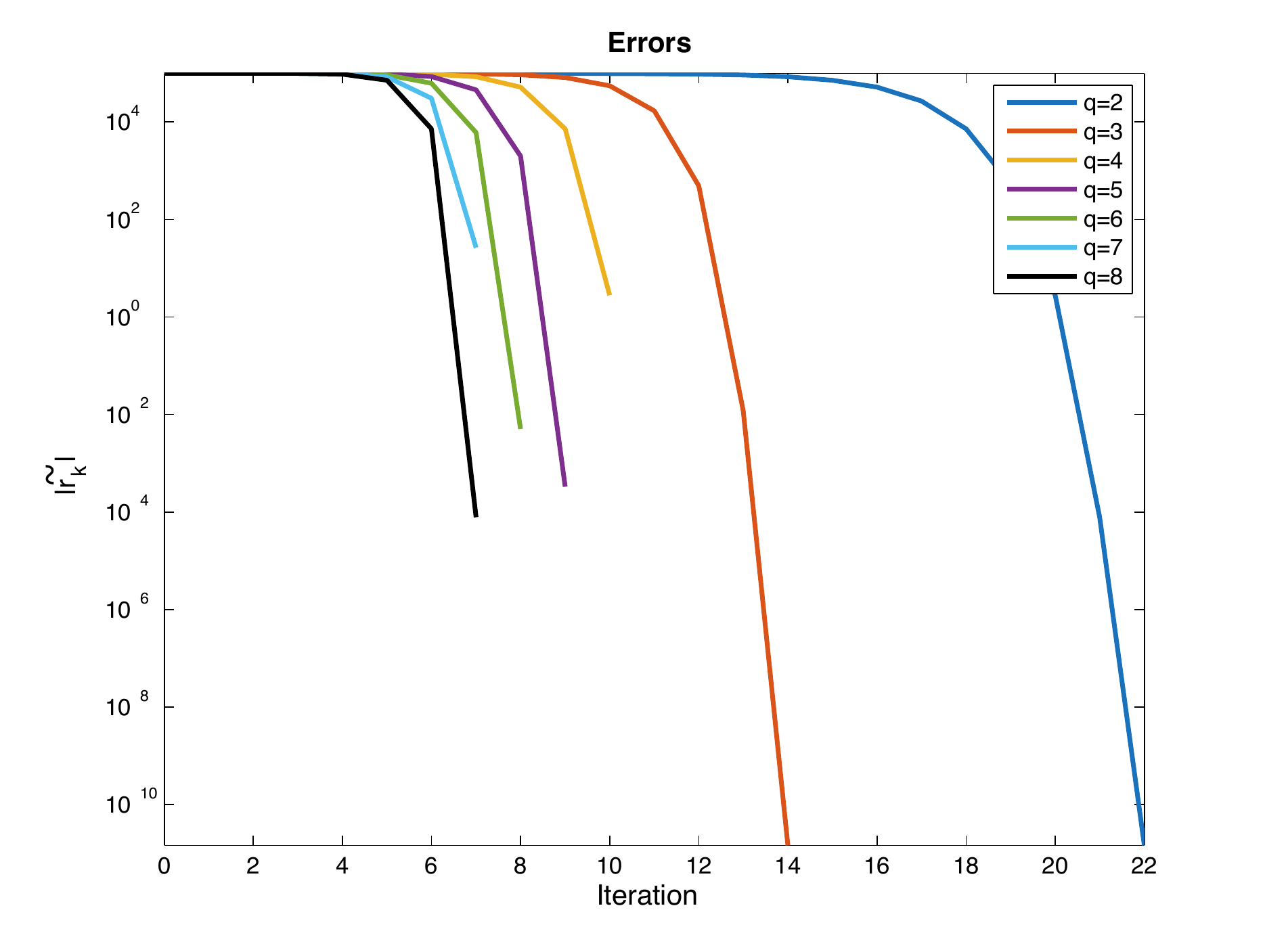}
\caption[Errors for $p=1$ and $\lambda = 10^{-6}$]{Errors for $p=1$ and $\lambda = 10^{-6}$. The optimal choice is $q=5$, even though $q=7$ and $q=8$ require \secrevision{fewer} iterations, but overall more multiplications.}
\label{fig:p1a1}
\end{figure}

As \revision{shown} in Fig.~\ref{fig:p1a1} for $p=1$, $\lambda = 10^{-6}$, a larger $q$ causes the iteration to \revision{enter convergence} after \secrevision{fewer} iterations. This is due to the fact that in every iteration 
\begin{align}
b_{k+1} =& \, \frac{1}{p} \left[ (p-1) - \left( (1-b_k^p\lambda)^q-1 \right)/(b_k^p\lambda) \right]  b_k \nonumber \\
=& \, b_k + \frac{1}{p} \left( \sum_{i=1}^{q-1} r_k^i \right) b_k \label{eq:sumrk}
\end{align}
it is calculated. \revision{Hence,} a larger $q$ leads to more summands in Eq.~\revision{(\ref{eq:sumrk})} and therefore to larger steps and variation of $b_{k+1}$. This also holds for negative $r_k$, as we have $r_k \in (-1,1)$ and therefore $|r_k^{i+1}| < |r_k^i|$.
\revision{This argument is also true for $p\neq1$ which explains the faster convergence for larger $q$ despite that the order of convergence is limited to $2$ in this case.}
However, a larger value for $q$ obviously increases the performance just up to a certain limit. This is due to the fact that a larger value of $q$ implies that $b^p_k$ is raised by larger exponents. Therefore, the number of multiplications increases, which is, especially in the matrix case, \revision{computationally the most time consuming part}. This is why we not only take into account the number of iterations but also the number of multiplications.

\subsection{The Matrix Case\label{sec:level4b}}
For evaluating the performance of our formula for matrices, we selected various set-ups with different variables $p$, $q$, densities $d$ and condition numbers \revision{$\kappa$}. The density of a matrix is defined as the number of its non-zero elements divided by its total number of elements. The condition number of a normal matrix, for which $AA^* = A^*A$ holds, is defined as the quotient of its eigenvalue with largest absolute value and its eigenvalue with smallest absolute value. Well-conditioned matrices \revision{have condition numbers close to $1$}. The bigger the condition number is, the more
ill-conditioned $A$ is. For each set-up, we take ten symmetric positive definite matrices $A \in \mathbb{R}^{1000 \times 1000}$ with random entries, generated by the MATLAB function \texttt{sprandsym} \cite{matlab}. This yields matrices with a spectral radius $\rho(A)<1$. \revision{We store the number of required iterations and the number of matrix-matrix multiplications for each random matrix}. Then, we average these values over all considered random matrices. \revision{We choose $q$ between $2$ and $6$ and stop the algorithm as soon as the norm of the residual is below $\varepsilon = 10^{-4}$.}

By using the sum representation of Eq.~\revision{(\ref{eq:def_pro})}, i.e.\revision{,}
\begin{align*} 
  B_{k+1} = \frac{1}{p} \left[ (p-1) I - \sum_{i=1}^q \binom{q}{i}(-1)^i(B_k^{p}A)^{i-1} \right]B_k, 
\end{align*} 
where the number of multiplications is minimized by saving previous powers of $B_k^pA$. It is evident that the number of matrix-matrix multiplications for a particular \revision{number of iterations} is minimal for the smallest value of $q$. However, a larger $q$ can also \revision{mean \secrevision{fewer} iterations. Consequently, the best value for $q$ usually is} the smallest that yields the lowest possible number of iterations for a certain set-up. In general, the number of matrix-matrix multiplications $m$ can be determined as a function of $p$, $q$, as well as the number of necessary iterations $j$ and which reads as 
\begeq{m(p,q,j) = p+\left((q-1)+p\right)j.}

To fulfill the conditions of Theorem \ref{them}, we claim that the start matrix $B_0$ commutes with the given matrix $A$. If $B_0$ is chosen as a positive multiple of the identity matrix or the matrix $A$ itself, i.e.\revision{,} $B_0 = \alpha I$ or $B_0 = \alpha A$ for $\alpha > 0$, respectively, then it is obvious that $AB_0 = B_0A$ holds and that $AB_0$ is symmetric positive definite.
\secrevision{For the range of $p$ and $q$ used in this evaluation, the simplified condition $\|R_0\|_2\le 1$ suffices to guarantee convergence as shown in Tab.~\ref{tab:overview}}.
Also, in the first part of the calculations, we only deal with matrices $A$ that have a spectral radius smaller than $1$. If we take $B_0 = I$, then we have $\norm{I-B_0A}_2<1$ due to the following
\begin{lemma}
\label{lemme}
Let $C \in \mathbb{C}^{n \times n}$ be a Hermitian positive definite matrix with $\norm{C}_2 \leq 1$. Then, \revision{it holds that $\norm{I-C}_2<1$.}
\end{lemma}
\begin{proof}
It is apparent that $\norm{I}_2 = 1$. Let $U$ be the unitary matrix such that $U^{-1}CU = \diag(\mu_1, \ldots, \mu_n)$, where $\mu_i \in (0,1] $ are the eigenvalues of $C$. Then we have
\begin{align*}
\norm{I-C}_2 = \norm{U^{-1}(I-C)U}_2 
= \norm{I-\diag(\mu_1, \ldots, \mu_n)}_2 
= 1-\mu_{\min} < 1,
\end{align*}
where $\mu_{\min}$ is the smallest eigenvalue of $C$.
\end{proof}

\begin{table}[bt]
\centering
  \begin{minipage}[t]{0.475\textwidth}
    \centering
 \caption{\revision{Numerical results for $\kappa=500$, $p=1$ and $d\in\{0.003, 0.1\}$. Optimal values are shown in bold.\vspace{5mm}}}
 \label{c500_1}
 \begin{tabular}{rrr}
\hline\hline
 $q$ & \#it & \#mult\Tstrut\Bstrut\\
 \hline
2 & 13 & 27\Tstrut\\
\textbf{3} & 8 & \textbf{25} \\
4 & 7 & 29 \\
5 & 6 & 31 \\
6 & \textbf{5} & 31\Bstrut\\
\hline\hline
\end{tabular}
\end{minipage}\begin{minipage}[t]{0.05\textwidth}\hfill\end{minipage}\begin{minipage}[t]{0.475\textwidth}
\centering
\caption{\revision{Numerical results for $\kappa=500$, $p=4$ and $d\in\{0.003, 0.1\}$. Optimal values are shown in bold.\vspace{5mm}}}
\label{c500_4}
\begin{tabular}{rrr}
\hline\hline
$q$ & \#it & \#mult\Tstrut\Bstrut\\
\hline
2 & 10 & 54\Tstrut\\
3 & 6 & 40 \\
\textbf{4} & \textbf{5} & \textbf{39} \\
5 & \textbf{5} & 44 \\
6 & \textbf{5} & 49\Bstrut\\
\hline\hline
\end{tabular}
\end{minipage}
\end{table}
In \revision{the following,} we want to investigate \revision{the relation between the number of iterations and $q$ for different $p$, densities $d$ and condition numbers $\kappa$}.
Specifically, we begin with sparse matrices $A$, where $\revision{\kappa}=500$.
\revision{In this case, the results are identical for densities $d=0.003$ and $d=0.1$.}

\begin{table}[bt]
\centering
  \begin{minipage}[t]{0.475\textwidth}
    \centering
 \caption{\revision{Numerical results for $\kappa=10$, $p=1$ and $d=0.003$. Optimal values are shown in bold.\vspace{5mm}}}
 \label{c10_1}
 \begin{tabular}{rrr}
\hline\hline
 $q$ & \#it & \#mult\Tstrut\Bstrut\\
 \hline
\textbf{2} & 7 & \textbf{15}\Tstrut\\
3 & 5 & 16 \\
4 & 4 & 17 \\
5 & \textbf{3} & 16 \\
6 & \textbf{3} & 19\Bstrut\\
\hline\hline
\end{tabular}
\end{minipage}\begin{minipage}[t]{0.05\textwidth}\hfill\end{minipage}\begin{minipage}[t]{0.475\textwidth}
\centering
\caption{\revision{Numerical results for $\kappa=10$, $p=4$ and $d=0.003$. Optimal values are shown in bold.\vspace{5mm}}}
\label{c10_4}
\begin{tabular}{rrr}
\hline\hline
$q$ & \#it & \#mult\Tstrut\Bstrut\\
\hline
2 & 6 & 34\Tstrut\\
\textbf{3} & \textbf{4} & \textbf{28} \\
4 & \textbf{4} & 32 \\
5 & \textbf{4} & 36 \\
6 & \textbf{4} & 40\Bstrut\\
\hline\hline
\end{tabular}
\end{minipage}
\end{table}

\begin{figure}[bt]
\centering
\includegraphics[width=.7\textwidth]{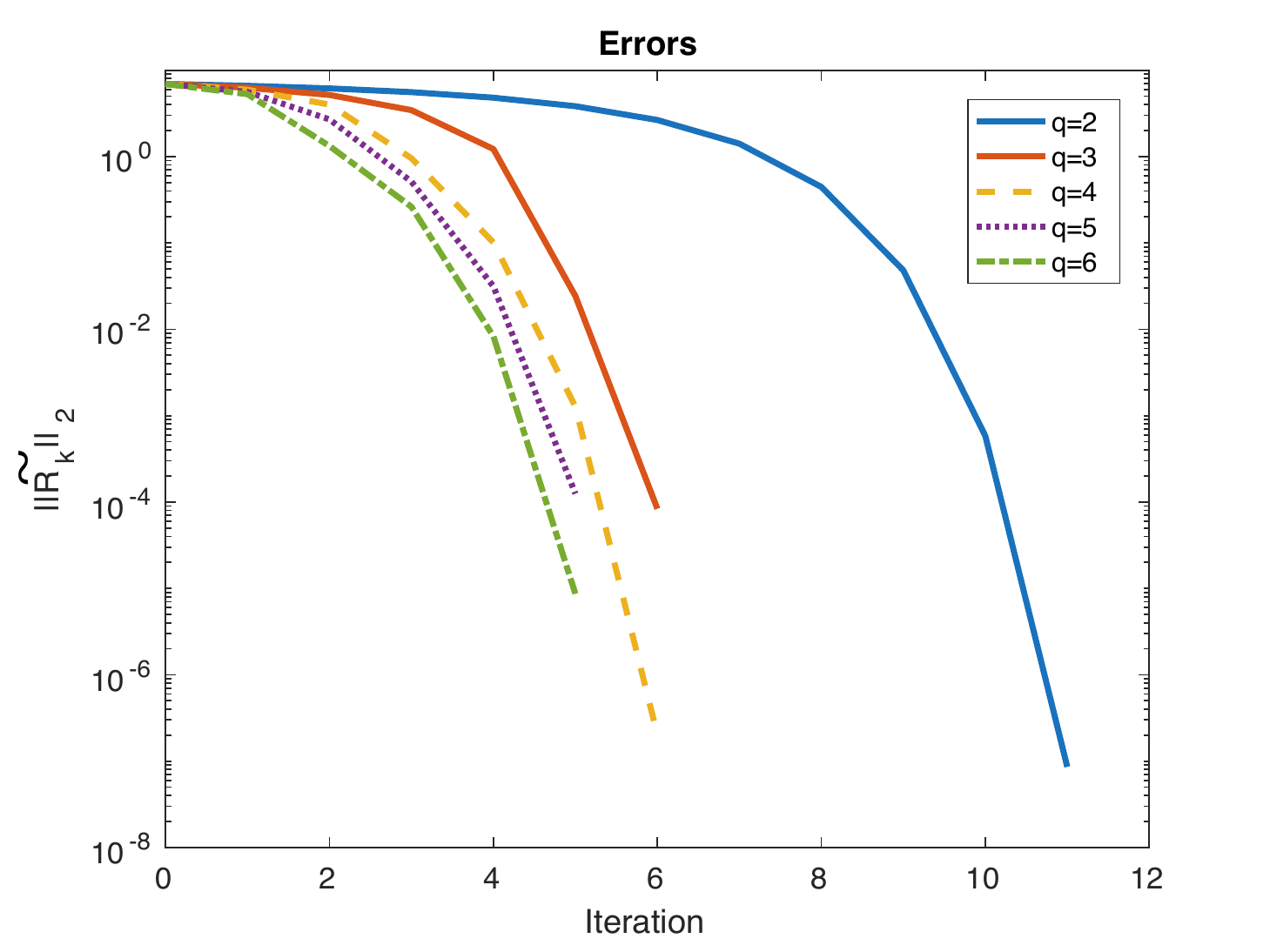}
\caption[Errors for $\revision{\kappa}=500$, $p=3$ and $d=0.003$]{Errors for $\revision{\kappa}=500$, $p=3$ and $d=0.003$.}
\label{fig:p3}
\end{figure}

As \revision{shown} in Table~\ref{c500_1}, for $p=1$, \revision{the number of iterations is lowest for $q=6$, but concerning the number of matrix-matrix multiplications $q=3$ is the best}.
\revision{The situation is similar for $p=2$ and $p=3$, where $q=3$ is optimal.}
As \revision{shown} in Table~\ref{c500_4}, for $p=4$, the number of iterations, \revision{and the number of matrix-matrix multiplications} are minimal for $q=4$.
\revision{This shows that for larger $p$, higher values for $q$ become beneficial.}
In general, like in the scalar case, a larger $q$ causes the iteration scheme to reach quadratic convergence after \secrevision{fewer} iterations as \revision{shown} in Fig.~\ref{fig:p3} for $p=3$.
\revision{Note that although the norm of the residual $R_k$ is used as a convergence criterion, in Fig.~\ref{fig:p3} we show the norm of the error compared to the correct solution $\tilde{R_k}$.}
\revision{For matrices with lower condition numbers, overall \secrevision{fewer} iterations are required as shown in Tables~\ref{c10_1} and~\ref{c10_4}. As a consequence, lower values for $q$ appear to be optimal. For $p=1$ the optimal value is $q=2$, for $p=4$ it is $q=3$.}

\begin{table}[tb]
  \centering
  \caption{\label{tab:cond}\revision{Achievable precision and number of required iterations for matrices with large condition numbers.}}
  \begin{tabular}{rrrrrr}
    \hline\hline
    $\kappa$ & $p$ & $q$ & final residual $\norm{R}_2$ & final error $\norm{\tilde R}_2$ & \#it\Tstrut\Bstrut\\
    \hline
    \multirow{4}{*}{\secrevision{$10^3$}} & \multirow{2}{*}{1} & 2 & \secrevision{$4.05\times10^{-14}$} & \secrevision{$2.86\times10^{-11}$} & 16\Tstrut\Bstrut\\
    \cline{3-6}
    && 6 & \secrevision{$4.08\times10^{-14}$} & \secrevision{$2.00\times10^{-11}$} & 7\Tstrut\Bstrut\\
    \cline{2-6}
    & \multirow{2}{*}{4} & 2 & \secrevision{$3.57\times10^{-05}$} & \secrevision{$7.77\times10^{-06}$} & 11\Tstrut\Bstrut\\
    \cline{3-6}
    && 6 & \secrevision{$1.88\times10^{-06}$} & \secrevision{$9.28\times10^{-09}$} & 6\Tstrut\Bstrut\\
    \hline\hline
    \multirow{4}{*}{\secrevision{$10^6$}} & \multirow{2}{*}{1} & 2 & \secrevision{$3.02\times10^{-11}$} & \secrevision{$1.17\times10^{-05}$} & 26\Tstrut\Bstrut\\
    \cline{3-6}
    && 6 & \secrevision{$1.67\times10^{-11}$} & \secrevision{$8.84\times10^{-06}$} & 11\Tstrut\Bstrut\\
    \cline{2-6}
    & \multirow{2}{*}{4} & 2 & \secrevision{$9.82\times10^{-01}$} & \secrevision{$2.00\times10^{+01}$} & 11\Tstrut\Bstrut\\
    \cline{3-6}
    && 6 & \secrevision{$3.62\times10^{-01}$} & \secrevision{$2.36\times10^{\pm00}$} & 5\Tstrut\Bstrut\\
    \hline\hline
    \multirow{4}{*}{\secrevision{$10^9$}} & \multirow{2}{*}{1} & 2 & \secrevision{$3.55\times10^{-08}$} & \secrevision{$3.48\times10^{+01}$} & 34\Tstrut\Bstrut\\
    \cline{3-6}
    && 6 & \secrevision{$2.42\times10^{-08}$} & \secrevision{$1.41\times10^{+01}$} & 15\Tstrut\Bstrut\\
    \cline{2-6}
    & \multirow{2}{*}{4} & 2 & \secrevision{$1.00\times10^{\pm00}$} & \secrevision{$1.66\times10^{+02}$} & 11\Tstrut\Bstrut\\
    \cline{3-6}
    && 6 & \secrevision{$1.00\times10^{\pm00}$} & \secrevision{$1.52\times10^{+02}$} & 4\Tstrut\Bstrut\\
    \hline\hline
  \end{tabular}
\end{table}

\revision{To evaluate the impact of much larger condition numbers, we used matrices with $\kappa\in\{\secrevision{10^3, 10^6, 10^9}\}$ for different combinations of $p$ and $q$. All calculations were performed using double precision arithmetic. The results are shown in Table~\ref{tab:cond}. As can be seen, the final precision decreases for higher condition numbers and the number of
required iteration increases. Using higher values for $q$ not only reduces the number of iterations, but also increases the final precision in all cases.}

\revision{Overall it shows that the densitiy of the matrix has no influence on the choice of an optimal $q$ but for higher $p$ and for higher condition number $\kappa$, increasing values of $q$ can provide a reduction not only in the number of iterations but also the number of matrix-matrix multiplications and increase the precision of the final result.}

\subsection{General Matrices and Applications}
For many applications in chemistry and physics, the most interesting are sparse
matrices with an arbitrary spectral radius. Often, to obtain an algorithm that scales linearly with the number of rows/columns of the relevant matrices, it is crucial to have sparse matrices.
However, typically these matrices do have a spectral radius that is larger than $1$, which is in contrast with the matrices we have considered so far, where $\rho(A) < 1$.

As already discussed, for $q=2$, the iteration function as obtained by Eq.~\revision{(\ref{eq:def_pro})} is identical to Bini's iteration scheme of Eq.~\revision{(\ref{eqn:q2p-1})}. Nevertheless, our numerical investigations suggest that the analogue of Proposition~\ref{prop} from Section~\ref{sec:level2a} may also hold for Eq.~\revision{(\ref{eq:def_pro})} for $q \neq 2$.
In any case, in order to also deal with matrices whose spectral radius is larger than $1$, one can either scale the matrix such that \revision{it} has a spectral radius $\rho(A) < 1$, or choose the matrix $B_0$ so that $\norm{I - B_0^pA}_{\revision{2}} < 1$ is satisfied.
With respect to the latter, for the widely-used Newton-Schulz iteration, employing 
\begin{align}
B_0 = (\norm{A}_1\norm{A}_{\infty})^{-1}A^\mathsf{T} \label{eq:pan}
\end{align}
as start matrix, convergence is guaranteed \cite{pan}. However, here we solve the problem for an arbitrary \revision{$p$} by the following

\begin{prop}
\label{pr}
Let $A$ be a symmetric positive definite matrix with $\rho(A) \geq 1$ and $B_0$ like in Eq.~\eqref{eq:pan}. Then, $\norm{I - B_0^pA}_2 < 1$ is guaranteed. 
\end{prop}
\begin{proof}
As $A$ is symmetric positive definite, we have $\norm{A}_2 = \lambda_{\max}$. Due to the fact that $\norm{A}_2 \leq \sqrt{\norm{A}_1 \norm{A}_{\infty}}$ and making use of Lemma~\ref{lemme}, we thus have to show that 
\begeq{\norm{B_0^pA}_2 = \norm{\left((\norm{A}_1\norm{A}_{\infty})^{-1}A^\mathsf{T}\right)^pA}_2 \leq 1.}
Since $A$ is symmetric, we have $A^\mathsf{T} = A$ and therefore commutativity of $B_0$ and $A$. Furthermore, the relation
\begeq{\left(\norm{A}_1 \norm{A}_{\infty} \right)^{-1} \leq \frac{1}{\lambda^2_{\max}}}
holds, so that eventually 
\begin{align*}
\norm{B_0^pA}_2 &\leq\frac{1}{(\lambda^2_{\max})^p}\norm{A^{p+1}}_2 \leq \frac{1}{\lambda^{2p}_{\max}} \norm{A}_2^{p+1} \nonumber \\
&= \frac{1}{\lambda^{2p}_{\max}}\lambda^{p+1}_{\max} = \frac{1}{\lambda^{p-1}_{\max}} \leq 1, 
\end{align*}
as we deal with matrices with $\rho(A) \geq 1$.
\end{proof}

\begin{rem}
By replacing $A^\mathsf{T}$ with $A^*$ in Eq.~\revision{(\ref{eq:pan})}, \textup{Proposition \ref{pr}} is also true for Hermitian positive definite matrices.
\end{rem}

\secrevision{Consequently, the initial guess of Eq.~\revision{(\ref{eq:pan})} is not only suited for the Newton-Schulz iteration, but for all cases in which the condition $\norm{I - B_0^pA}_2 < 1$ suffices to guarantee convergence (cf. Tab.~\ref{tab:overview}).}

In what follows, we perform calculations using matrices $A$ that have the same densities and condition numbers as in the case $\rho(A) < 1$. We scale the matrices such that $\rho(A) \in \{10, 50\}$ and use the initial guess of Eq.~\revision{(\ref{eq:pan})}.
\revision{In contrast to the results from Sec.~\ref{sec:level4b}, we see a slight effect of the density $d$ on the optimal choice of $q$. Hence, we evaluate the influence of different densities as well}.

As an example, we study matrices \revision{with} $\revision{\kappa}=500$ \revision{and} spectral radius $\rho(A) = 10$ for the case $p=3$. As \revision{shown} in Table~\ref{rho10}, the number of matrix-matrix multiplications is always the lowest for $q=5$, even though the number of iterations is minimal for $q=6$.
\begin{table}[bt]
\centering
  \begin{minipage}[t]{0.475\textwidth}
\centering
\caption{Numerical results for the case $p=3$, $\rho(A)=10$ and $\revision{\kappa}=500$. Optimal values are shown in bold.\vspace{5mm}}
\label{rho10}
\begin{tabular}{crrr}
\hline\hline
$d$ & $q$ & \#it & \#mult\Tstrut\Bstrut\\
\hline
& 2 & 55 & 223\Tstrut\\
& 3 & 23 & 118 \\
0.003 & 4 & 18 & 111 \\
& \textbf{5} & 15 & \textbf{108} \\
& 6 & \textbf{14} & 115\Bstrut\\
\hline
& 2 & 56 & 227\Tstrut\\
& 3 & 23 & 118 \\
0.01 & 4 & 18 & 111 \\
& \textbf{5} & 15 & \textbf{108} \\
& 6 & \textbf{14} & 115\Bstrut\\
\hline
& 2 & 57 & 231\Tstrut\\
& 3 & 25 & 128 \\
0.1 & 4 & 19 & 117 \\
& \textbf{5} & 16 & \textbf{115} \\
& 6 & \textbf{15} & 123\Bstrut\\
\hline
& 2 & 62 & 251\Tstrut\\
& 3 & 27 & 138 \\
0.8 & 4 & 21 & \textbf{129} \\
& \textbf{5} & 18 & \textbf{129} \\
& 6 & \textbf{16} & 131\Bstrut\\
\hline\hline
\end{tabular}
\end{minipage}\begin{minipage}[t]{0.05\textwidth}\hfill\end{minipage}\begin{minipage}[t]{0.475\textwidth}
\centering
\caption{Numerical results for the case $p=4$, $\rho(A)=50$ and $\revision{\kappa}=10$. Optimal values are shown in bold.\vspace{5mm}}
\label{rho50}
\begin{tabular}{crrr}
\hline\hline
$d$ & $q$ & \#it & \#mult\Tstrut\Bstrut\\
\hline
& 2 & 32 & 164\Tstrut\\
& 3 & 18 & 112 \\
0.003 & 4 & 14 & 102 \\
& \textbf{5} & 12 & \textbf{100} \\
& 6 & \textbf{11} & 103\Bstrut\\
\hline
& 2 & 34 & 174\Tstrut\\
& 3 & 19 & 118 \\
0.01 & 4 & 15 & 109 \\
& \textbf{5} & 13 & \textbf{108} \\
& 6 & \textbf{12} & 112\Bstrut\\
\hline
& 2 & 37 & 189\Tstrut\\
& 3 & 21 & 130 \\
0.1 & 4 & 16 & 116 \\
& 5 & 14 & 116 \\
& \textbf{6} & \textbf{12} & \textbf{112}\Bstrut\\
\hline
& 2 & 43 & 219\Tstrut\\
& 3 & 24 & 148 \\
0.8 & 4 & 18 & \textbf{130} \\
& 5 & 16 & 132 \\
& \textbf{6} & \textbf{14} & \textbf{130}\Bstrut\\
\hline\hline
\end{tabular}
\end{minipage}
\end{table}
However, the decision of an optimal $q$ is not always as simple as in the case demonstrated above. One such instance is \revision{shown in Table~\ref{rho50}, where the inverse fourth root of matrices with spectral radius $\rho(A) = 50$ and condition number $\revision{\kappa}=10$ is computed}. In this case, we observe that for sparse matrices the choice $q=5$ is the best, while for more dense matrices $q=6$ is optimal.

For matrices with a larger spectral radius, we did not observe any configuration where $q=2$ is the best choice. On the contrary, as can be seen in Tables \ref{rho10} and \ref{rho50}, the evaluation of the inverse $p$-th root with $q=2$ generally requires up to two times the computational effort and number of \revision{matrix-matrix multiplications}, as well as up to three times the number of iterations with respect to the optimal choice of $q$.
In general, the larger $p$, the more profitable is a bigger value for $q$.

\revision{Although our initial value from Eq.~(\ref{eq:pan}) allows us to
operate on matrices with spectral radius $\rho(A)\geq1$, we observe that
significantly more iterations are required than for the matrices evaluated in
Sec.~\ref{sec:level4b}. It should be noted that the alternative approach by
scaling the matrices such that $\rho(A)<1$ and then running the iterative method
with $B_0 = I$ would require \secrevision{fewer} iteration in these test cases, but instead
requires the calculation of $\rho(A)$ and a proper scaling factor.}

\section{Conclusion}
We presented a new general algorithm to construct iteration functions to calculate the inverse principal $p$-th root of symmetric positive definite matrices. It includes, as special cases, the methods of Altman \cite{altm} and Bini et al. \cite{bini}. The variable $q$, that in Altman's work equals the order of convergence for the iterative inversion of matrices, in our scheme represents the order of expansion. We find that $q>2$ does not lead to a higher order of convergence if $p\neq1$, but that the iteration converges within \secrevision{fewer} iterations and matrix-matrix multiplications, as quadratic convergence is reached faster. For an optimally chosen value for $q$, the computational effort and the number of matrix-matrix multiplications is up to two times lower, and the number of iterations up to \revision{four} times lower \revision{compared} to $q=2$.

\section*{Acknowledgments}
This work was supported by the Max Planck Graduate Center with the Johannes Guten\-berg-Universit\"at Mainz (MPGC) and the IDEE project of the Carl Zeiss Foundation. The authors thank Martin Hanke-Bourgeois for valuable suggestions and Stefanie Hollborn for critically reading the manuscript. This project has received funding from the European Research Council (ERC) under the European Union's Horizon 2020 research and innovation programme (grant agreement No 716142).

\bibliographystyle{unsrt}
\bibliography{paper}

\end{document}